\documentclass{amsart}
\usepackage{verbatim, amscd,
  epsfig,amsmath,amsthm,amstext,amssymb,todonotes,float}
\usepackage[colorlinks, citecolor=red, breaklinks=true]{hyperref} %makes every reference a link in the PDF file
\usepackage[all]{xy} 

\let\preAccentstilde\tilde
\let\preAccentsbar\bar
\usepackage{accents}
\input xypic

\usepackage{mathtools}
\usepackage{graphicx}

\usepackage[initials,msc-links,non-compressed-cites,nobysame]{amsrefs}[2007/10/22]
\newcommand{\TitleWithUrl}[1]{\IfEmptyBibField{doi}%
  {\IfEmptyBibField{url}{\textit{#1}}%
    {\IfEmptyBibField{eprint}{\href {\BibField{url}}{\textit{#1}}}{\textit{#1}}}%
    }%
  {\href {https://doi.org/\BibField{doi}}{\textit{#1}}}}
\renewcommand{\eprint}[1]{\IfEmptyBibField{url}{\url{#1}}%
  {\href {\BibField{url}}{#1}}}

\BibSpec{article}{%
    +{}  {\PrintAuthors}                {author}
    +{,} { \TitleWithUrl}               {title}
    +{.} { }                            {part}
    +{:} { \textit}                     {subtitle}
    +{,} { \PrintContributions}         {contribution}
    +{.} { \PrintPartials}              {partial}
    +{,} { }                            {journal}
    +{}  { \textbf}                     {volume}
    +{}  { \PrintDatePV}                {date}
    +{,} { \issuetext}                  {number}
    +{,} { \eprintpages}                {pages}
    +{,} { }                            {status}
    +{,} { available at arXiv:\eprint}        {eprint}
    +{}  { \parenthesize}               {language}
    +{}  { \PrintTranslation}           {translation}
    +{;} { \PrintReprint}               {reprint}
    +{.} { }                            {note}
    +{.} {}                             {transition}
    +{}  {\SentenceSpace \PrintReviews} {review}
}

\BibSpec{collection.article}{%
    +{}  {\PrintAuthors}                {author}
    +{,} { \TitleWithUrl}                     {title}
    +{.} { }                            {part}
    +{:} { \textit}                     {subtitle}
    +{,} { \PrintContributions}         {contribution}
    +{,} { \PrintConference}            {conference}
    +{}  {\PrintBook}                   {book}
    +{,} { }                            {booktitle}
    +{,} { \PrintDateB}                 {date}
    +{,} { pp.~}                        {pages}
    +{,} { }                            {status}
    +{,} { available at \eprint}        {eprint}
    +{}  { \parenthesize}               {language}
    +{}  { \PrintTranslation}           {translation}
    +{;} { \PrintReprint}               {reprint}
    +{.} { }                            {note}
    +{.} {}                             {transition}
    +{}  {\SentenceSpace \PrintReviews} {review}
}
\BibSpec{book}{%
    +{}  {\PrintPrimary}                {transition}
    +{,} { \TitleWithUrl}               {title}
    +{.} { }                            {part}
    +{:} { \textit}                     {subtitle}
    +{,} { \PrintEdition}               {edition}
    +{}  { \PrintEditorsB}              {editor}
    +{,} { \PrintTranslatorsC}          {translator}
    +{,} { \PrintContributions}         {contribution}
    +{,} { }                            {series}
    +{,} { \voltext}                    {volume}
    +{,} { }                            {publisher}
    +{,} { }                            {organization}
    +{,} { }                            {address}
    +{,} { \PrintDateB}                 {date}
    +{,} { }                            {status}
    +{}  { \parenthesize}               {language}
    +{}  { \PrintTranslation}           {translation}
    +{;} { \PrintReprint}               {reprint}
    +{.} { }                            {note}
    +{.} {}                             {transition}
    +{}  {\SentenceSpace \PrintReviews} {review}
}

\BibSpec{thesis}{%
    +{}  {\PrintAuthors}                {author}
    +{.} { \TitleWithUrl}               {title}
    +{:} { \textit}                     {subtitle}
    +{,} { \PrintThesisType}            {type}
    +{,} { }                            {organization}
    +{} { \PrintDate}                  {date}
    +{,} { }                            {address}
    +{,} { \eprint}                     {eprint}
    +{,} { }                            {status}
    +{}  { \parenthesize}               {language}
    +{}  { \PrintTranslation}           {translation}
    +{;} { \PrintReprint}               {reprint}
    +{.} { }                            {note}
    +{.} {}                             {transition}
    +{}  {\SentenceSpace \PrintReviews} {review}
}

\theoremstyle{plain}
\newtheorem*{theorem*}{Theorem}
\newtheorem{theorem}{Theorem}[section]

\newtheorem{cor}[theorem]{Corollary}
\newtheorem{prop}[theorem]{Proposition}

\theoremstyle{definition}
\newtheorem{definition}[theorem]{Definition}

\theoremstyle{remark}
\newtheorem{rem}[theorem]{Remark}

\numberwithin{equation}{section}

\renewcommand{\Re}{{\rm Re}\,}

\newcommand{\R}{\mathbb{ R}}
\newcommand{\C}{\mathbb{ C}}
\newcommand{\Z}{\mathbb{ Z}}

\renewcommand{\H}{\mathbb{ H}}

\newcommand{\N}{\mathbb{ N}}

\newcommand{\ttrivial}[1]{\underline{\widetilde\H}^{#1}}
\newcommand{\invers}{^{-1}}

\DeclareMathOperator{\ImQ}{Im}
\DeclareMathOperator{\Span}{span}

\DeclareMathOperator{\dn}{dn}
\DeclareMathOperator{\am}{am}
\DeclareMathOperator{\sn}{sn}

\newlength{\dhatheight}

\begin{document}

\title
{New explicit CMC cylinders and same--lobed CMC multibubbletons}

\author{Joseph Cho}
\address[Joseph Cho]{Institute of Discrete Mathematics and Geometry, TU Wien, Wiedner Hauptstrasse 8-10/104, 1040 Wien, Austria}
\email{jcho@geometrie.tuwien.ac.at}

\author{Katrin Leschke}
\address[Katrin Leschke]{Department of Mathematics,
  University of Leicester, University Road, Leicester LE1 7RH, United
  Kingdom\newline
Research Institute for Mathematical Sciences, Kyoto University, Kyoto, Japan}
\email{k.leschke@leicester.ac.uk }

\author{Yuta Ogata}
\address[Yuta Ogata]{Department of Mathematics, Faculty of Science, Kyoto Sangyo University, 
  Motoyama, Kamigamo, Kita-ku, Kyoto-City, 603-8555, Japan.
 }
\email{yogata@cc.kyoto-su.ac.jp}


%\date{\today}

\begin{abstract} We provide explicit parametrisations of all Darboux transforms of Delaunay surfaces. 
 Using the Darboux transformation on a multiple cover, we obtain this way new closed CMC surfaces with dihedral
 symmetry. These can be used to construct closed same-lobed CMC
 multibubbletons by applying Bianchi permutability.
\end{abstract}

%\thanks{}
\maketitle

\section{Introduction}

In this paper, we investigate Darboux transforms of Delaunay surfaces,
that is, of constant mean
curvature surfaces (CMC) of revolution with
non--vanishing mean curvature. 
Recall that the \emph{classical Darboux transformation} \cite{darboux}
is defined for isothermic surfaces, those surfaces which allow
a conformal curvature line parameterisation. Geometrically, a Darboux
pair is given by a pair of isothermic surfaces, both conformally
enveloping a sphere congruence. Since a CMC surface $f: M \to\R^3$ is isothermic, one
can apply the classical Darboux transformation; however, a classical
Darboux transform $\hat f$ is in general just isothermic, and has
constant mean curvature only if $\hat f$ has pointwise constant
distance from the parallel constant mean curvature surface of $f$, see
\cite{darboux_isothermic}. We will denote these Darboux transforms for
short by \emph{CMC Darboux transforms}. CMC Darboux transforms are equivalent to the
Bianchi-Bäcklund transforms \cite{bianchi_lezioni_1903} and
simple-factor dressings of the extended frame
\cite{kobayashi_characterizations_2005, burstall_isothermic_2006,
simple_factor_dressing,  cho_simple_2019}.

We will construct CMC Darboux transforms via the $\mu$--Darboux
transformation for CMC surfaces: this transformation is a special case
of the (generalised) Darboux transformation for conformal immersions, 
 \cite{conformal_tori}.  We briefly recall the 
$\mu$--Darboux transformation (see also Section~\ref{sec:background}):  By the Ruh--Vilms
Theorem \cite{ruh_vilms} a surface has constant mean curvature if and
only if its Gauss map is harmonic. The harmonicity of the Gauss map
allows for an introduction of an associated family of flat complex connections $d_\lambda$,
$\lambda\in\C_* =\C\setminus\{0\}$.  Then
the $\mu$--Darboux transforms of a CMC surface are exactly given by the
$d_\mu$--parallel sections for some fixed \emph{spectral parameter}  $\mu\in\C_*$. In the case when
$\mu\in\R_*=\R\setminus\{0\}$, this construction  gives exactly all
(classical) CMC Darboux
  transforms, \cite{cmc}.

In the case of CMC Darboux
transforms of  Delaunay surfaces, for a
generic spectral parameter $\mu\in\R_*$ the only periodic 
CMC Darboux transforms are trivial, that is the Darboux transforms are again
reparametrisations, rotations and translations of the original
Delaunay surface.

In case of a round cylinder however Sterling and Wente
\cite{wente_sterling} discovered, while using the Bianchi--B\"{a}cklund
transform, that there are special spectral parameters, the so--called
\emph{resonance points}, so that there are closed, non--trivial CMC
Darboux transforms; these surfaces are called (CMC) \emph{bubbletons}
and are not embedded \cite{Kilian_bubbletons}.  For each $n\in\N, n>1,$
there is a corresponding resonance point
$\mu_n $, and $n$ gives the number of lobes
on the bubble of the bubbleton in case of a single cover of the surface.  This process
has been extended \cite{KobayashiBubbletons, KobayashiBubbletons2} to obtain bubbletons from
any Delaunay surface (and in all space forms).

From our point of view (Section~\ref{sect:dbubble}),  closed CMC Darboux transforms are
characterised by parallel sections \emph{with multiplier}: going
around the circle direction, the parallel section only changes by a
complex multiple. Resonance points are then spectral parameter
$\mu\in\R\setminus\{0, \pm 1\}$ for which all parallel sections are
sections with multiplier. In the case of Delaunay surfaces we first
provide an explicit parameterisation of all Darboux transforms in
terms of Jacobi elliptic functions and the  incomplete elliptic
integral of the third kind by computing parallel sections in a
quaternionic formalism (see Theorem~\ref{thm:parallel explicit}).
We then investigate  the parallel sections with multipliers, and in particular
obtain  resonance points $\mu_n$ which depend
on the necksize $r$ of the Delaunay surface. 
% in dependence of 
% \[ \mu_n^\pm =\frac{1}{2r(1-r)}\left(n^2-1 +2r(1-r) \pm
%     \sqrt{(n^2-1)^2+ 4r(1-r)(n^2-1)}\right),\]
% for $n \in \N, n>1$, where $r\le \frac 12, r\not=0,$ is the necksize
% of the Delaunay surface.
For necksize
$r=\frac 12$ we have the round cylinder, for $0<r<\frac 12$ an
unduloid and $r<0$ a nodoid. As before $n\in\N, n>1$, gives the number
of lobes if the Delaunay surface is embedded. 
 In the case of nodoids, that
is, non--embedded Delaunay surfaces, the necksize $r<0$ imposes the 
condition $n>1-2r$  on the number of lobes on a bubbleton. 
%$\mu_n^\pm\in\R\setminus\{0,\pm 1\}$.

Given two Darboux transforms $f_1$ and $f_2$ 
one can algebraically construct a common Darboux transform $\hat f$
via Bianchi permutability, \cite{bianchi_ricerche_1905, conformal_tori}.  Applying
this construction in our setting, we obtain a \emph{doublebubbleton} for two resonance points
$\mu_n, \mu_l\in\R_*$, $n\not=l, n, l\in\N$, and corresponding
bubbletons $f_1$ and $f_2$ with $n$ and $l$ lobes respectively, 
 that is, a common closed Darboux transform $\hat f$
of $f_1$ and $f_2$ which has constant mean curvature. Such a
doublebubbleton has two bubbles, one with $n$ lobes, the other with
$l$ lobes. For this construction the assumption $n\not=l$ is critical:
for $\mu_n =\mu_l$ the common CMC Darboux transform obtained by Bianchi
permutability is the original Delaunay surface. Repeating the
procedure we can obtain multibubbletons with an arbitrary number of
bubbles but each two bubbles have different numbers of lobes.
Put differently, with
this approach we cannot obtain closed, same--lobbed CMC
multibubbletons.

On the other hand, we showed in \cite{sym-darboux} by using a
Sym--type argument 
that any closed classical double Darboux transform of a Delaunay surface with the
same spectral parameter can only have constant mean curvature  if it is the
original surface.  Therefore, if we
restrict to a single cover of the Delaunay surface, no same--lobed
closed CMC multibubbletons can occur.

We now consider multiple covers of Delaunay
surfaces (Section~\ref{sect:newcmc}). This is motivated by results on Darboux transforms of
isothermic tori \cite{holly_tori}, circletons \cite{kilian_circletons} and
periodic smooth and discrete Darboux transforms \cite{periodic_discrete}: in
these cases (additional) resonance points can be found when
considering multiple covers of the domain.

In the case of a multiple--cover of the round cylinder already Sterling
and Wente \cite{wente_sterling} observed that additional resonance
points $\mu_n^m$ can occur: one
obtains again a closed CMC surface with $n$ lobes on the $m$--fold cover of
the round cylinder where $n$ must be bigger than $m$.  This is a
natural generalisation of the situation on a single
cover. Indeed, in this paper we show that for nodoids and unduloids, one can also
obtain additional resonance points $\mu_n^m$  for $n>m$.
Geometrically, these bubbletons close on an $m$--fold cover and  have non--embedded ends but
look similar to the ``normal'' bubbletons with $n$ lobes on a single
bubble (if $m$ and $n$ are co--prime). As before, in the case of
nodoids there is an obstruction on the number of lobes $n> m(1-2r)$. 

However, the situation is more intriguing for unduloids and nodoids:
in this situation %the resonance points $\mu_n^m$  are given by $m, n\in\N, m,n\ge 1$% :
% \[
%   \mu_{n}^m=\frac{1}{2r(1-r)}\left(\frac{n^2-m^2}{m^2} +2r(1-r) \pm
%   \sqrt{\frac{n^2-m^2}{m^2}\left(\frac{n^2-m^2}{m^2}+
%       4r(1-r)\right)}\right),
% \]
% and
the case $n< m$ can occur  (Theorem~\ref{thm:fullBif}). Geometrically, now
the ``bubbles'' are not localised anymore but the number $n$
still gives the dihedral symmetry of the surface. In contrast to the case of a single
cover, there is a restriction on the order of the symmetry for
unduloids $n<m(1-2r)$ whereas for any choice of $n<m$ we obtain a
bubbleton of a nodoid. In particular, for all nodoids  (and unduloids
with $0<r< \frac{m-1}{2m}$)   it is possible
to obtain a CMC surface with $n=1$, which just has a reflectional
symmetry and one ``lobe''.

To obtain closed, same--lobed CMC multibubbletons, we change our point
of view (Section~\ref{sect:same}). Limiting to the case when $n>m$ (to obtain localised bubbles
with lobes on our bubbletons),  we investigate multibubbletons: if there
are two co--prime pairs $(m_1, n)$ and $(m_2, n)$, $m_1\not=m_2$, so
that the corresponding Darboux transforms $f_1$ and $f_2$ have $n$
bubbles on the $m_1$ and $m_2$--fold cover respectively, then the
common Darboux transform of $f_1$ and $f_2$ given by Bianchi
permutability  has two bubbles, each with
$n$ lobes --- we obtain a CMC same--lobed doublebubbleton which closes
on an $L$--fold cover, where $L$ is the least common multiple of
$m_1$ and $m_2$.

Finally, we observe that the number $1\le l_0\le m$ of resonance
points given by $n$, that is the number of $n$--lobed bubbles we can
put on a Delaunay surface,   depends on the necksize of the
Delaunay surface and co--prime pairs $(m,n)$: for any $l\le l_0$ we
show that there exists a multibubbleton
with $l$ bubbles with $n$ lobes each (Theorem~\ref{thm:samecmc}).

\quad

{\bf Acknowledgements.} The second author would like to thank the
  members of the
  Research Institute for  Mathematical Sciences at Kyoto University
  for their hospitality during her stay as Visiting Professor at the
  Institute. The third author gratefully acknowledges the
  support from Grants-in-Aid of JSPS Research Fellowships for Young
  Scientist 21K13799. 

\section{Background}
\label{sec:background}

In this section we will give a short summary of results and methods
used in this paper. For details on the quaternionic formalism and
CMC surfaces we refer to  \cite{coimbra, klassiker,
  cmc, simple_factor_dressing, darboux_isothermic}.

\subsection{CMC surfaces in the quaternionic model}

In this paper, we investigate the Darboux transformation on
conformal immersions $f: M \to\R^3$ from a Riemann surface into
3--space with constant mean curvature, where we assume
without loss of generality that the mean curvature is $H=1$. By the
Ruh--Vilms theorem \cite{ruh_vilms} the Gauss map $N$ of  a CMC
surface $f$ is 
harmonic. This allows to introduce a spectral parameter and to define,
for example, 
the associated family of flat connections, the associated family of
CMC surfaces, and the spectral curve of a CMC torus,
\cite{hitchin-harmonic, pin&ster, Bob}.

Here we use a quaternionic model and identify 3--space with
the imaginary quaternions $\R^3=\ImQ\H$ where
$\H =\Span_\R\{1, i, j, k\}$ and $i^2=j^2=k^2=ijk=-1$. In the
following we will  identify
\[\H =\Re \H \oplus \ImQ\H =\R\oplus \R^3.
\]
For imaginary
quaternions the product in the quaternions is related to the inner product
$\langle\cdot, \cdot\rangle$ and the cross product in $\R^3$ by
\[
ab =-\langle a,b\rangle +a\times b, \quad a, b\in\ImQ\H.
\]
 In
particular, we see
\[
S^2 =\{ n\in\ImQ\H \mid n^2=-1\}.
\]
Thus, if $f:  M \to\R^3$ is an immersion then its Gauss map $N: M \to
S^2$ is a complex structure $N^2=-1$ on $\R^4=\H$. Moreover, if $(M, J_{TM})$ is a
Riemann surface, where    $J_{TM}$ is the complex structure on the tangent space, then $f: M \to \R^3$ is conformal if and only if
\[
*df  = N df = -df N,
\]
where $*$ denotes the negative Hodge star operator, that is,
$*\omega(X) = \omega(J_{TM}X)$ for $X\in TM$,
$\omega\in\Omega^1(M)$.

Denoting the $(1,0)$--part ($(0,1)$-part, respectively) of a 1--form $\omega$ with respect
to the complex structure $N$ by
\[
  \omega' = \frac 12(\omega -
  N*\omega)
\]
(and $\omega''=\frac 12(\omega+N*\omega)$, respectively), we have for any conformal immersion $f: M
\to\R^3$ that
\begin{equation}
  \label{eq:H}
  df H = -(dN)'.
\end{equation}
In the case when $H=1$, the decomposition of $dN$ into $(0,1)$--
and $(1,0)$--part is thus given by 
\[
  dN = dg-df
\]
where $dg =(dN)''$ and $df=-(dN)'$. Put differently, $g= f+N$ is the
parallel CMC surface of $f$ since $N_g = -N$ and thus $dg = -(dN_g)'$.

In particular, we obtain the following reformulation of the Ruh--Vilms
theorem: a conformal immersion $f$ has constant mean curvature $H$ if
and only if $N$ is harmonic, that is, 
\[
  d((dN)')=0.
\]
Harmonicity allows for the introduction of  a complex spectral
parameter (for details in our setting, see
\cite{cmc}) to obtain the family of connections $d_\lambda$, $\lambda\in\C_* =
\C\setminus\{0\}$, given by
\begin{equation}
  \label{eq:associated family}
  d_\lambda \alpha = d\alpha -\frac 12(dN)'(N\alpha(a-1)+\alpha b)
\end{equation}
where $a=
\frac{\lambda+\lambda\invers}2, b=i
\frac{\lambda\invers-\lambda}2$.

The associated family can now be used to characterise CMC surfaces:
the connections $d_\lambda$ of the associated family (\ref{eq:associated family}) of a map $N: M \to S^2$
are flat for all $\lambda\in\C_*$ if and only if $N$ is harmonic.

\begin{rem}
  In case when $\lambda = a+ib\in S^1$ we have
$a, b\in\R$ and the family of flat connections is quaternionic.  
In general, the family of flat connections is defined on the trivial
$\C^2$--bundle $M\times\C^2= \underline{\C}^2$ over $M$ where we
consider $\C^2 = (\H, I)$ with the complex structure $I$ given by
right multiplication by $i$. In this case, we have the reality
condition
\begin{equation}
  \label{eq:reality}
  d_\lambda(\alpha j) = (d_{\bar\lambda\invers}\alpha)j.
\end{equation}
\end{rem}

\subsection{CMC surfaces and Darboux transforms}

We are now in the position to recall the Darboux transformation of CMC
surfaces. In \cite{conformal_tori} a Darboux transformation on
conformal immersions $f: M \to\R^3$ was defined, generalising the
classical Darboux transformation in \cite{darboux}. In this paper we
consider the so--called $\mu$--Darboux transformation on CMC surfaces,
which gives the $\mu$--Darboux transforms by the choice of a spectral
parameter $\mu\in\C_*$ and a $d_\mu$--parallel section. The
$\mu$--Darboux transformation is a special case of the generalised
Darboux transformation and preserves the CMC property, up to
translation. Additionally, the $\mu$--Darboux transforms with
$\mu\in\R\setminus\{0,1\}$ are exactly the classical Darboux
transforms of a CMC surface which are CMC. In this paper, we
investigate these \emph{(classical) CMC Darboux transforms} in the case of Delaunay
surfaces.  For a detailed study of the
link between these different Darboux transformations, we refer to
\cite{isothermic_paper}. 

We recall:
\begin{theorem}[\cite{cmc}]
  \label{thm:mudtcmc}
  Let $f: M\to\R^3$ be a CMC surface with $H=1$, and let $d_\lambda$
  be its associated family of flat connections.  For $\mu\in\R\setminus\{0,1\}$, let
  $\alpha\in\Gamma(\ttrivial{})$ be a $d_\mu$--parallel section of the
  trivial $\C^2$--bundle $\tilde M \times \C^2 = \tilde M \times \H$ over the universal cover $\tilde M$ of $M$.

  Then $\hat f=  f+ T$ is a CMC Darboux transform of $f$ where
  $T$ is given algebraically by
  \[
    T\invers = \frac12(N(a-1) + \alpha b \alpha\invers),
  \]
  and $a=\frac{\mu +\mu\invers}2, b= i\frac{\mu\invers-\mu}2$. 
  All classical CMC Darboux transforms of $f$  arise this
  way. 
\end{theorem}

\begin{rem}
  \label{rem:double}
  Note that $\mu = a+ i b$ and therefore $\mu\in\R$ if
  and only if $a \in\R, b\in i\R$. In this case, the condition
  $a^2+b^2=1$ shows that $\mu\invers = a-ib$. Therefore, by the reality condition
  (\ref{eq:reality})   $\alpha$ is
  $d_{a+ib}$--parallel if and only if $\alpha j$ is
  $d_{a-ib}$--parallel. In particular, in the case of real parameter,
  the two parameter $\mu = a+ ib$ and $\mu\invers= a-ib$ give rise
  to the same Darboux transforms.

  Moreover, for $\mu=-1$, that is,
  $a=-1, b=0$, the Darboux transform is independent of the parallel
  section $\alpha$ and  we obtain the parallel CMC surface $\hat f = f+N$ as
  the only Darboux transform.
\end{rem}

We now recall the Bianchi permutability theorem
\cite{bianchi_ricerche_1905, conformal_tori}  in our setting which
gives two--step Darboux transforms algebraically, that is, without
further integration, from the parallel sections of the first CMC
surface. 
 
\begin{theorem}[\cite{isothermic_paper}]
  Let $f: M\to\R^3$ be a CMC surface and $\alpha_l\in\Gamma(\ttrivial
  {})$ parallel sections with respect to the flat connections
  $d_{\mu_l}$ where $\mu_l\in\R\setminus\{0,1\}$, $l=1,2$, $\mu_1\not=\mu_2$. Let $f_l = f+
  \alpha_l\beta_l\invers$ be the CMC Darboux transforms given by
  $\alpha_l$ 
  and
  \begin{equation}
    \label{eq:beta}
    \beta_l=\frac 12(N\alpha_l(a_l-1) + \alpha_l b_l)
  \end{equation}
  where $a_l=\frac{\mu_l +\mu_l\invers}2,
  b_l=i\frac{\mu_l\invers-\mu_l}2$.
  
  Then the common $\mu$--Darboux transform $\hat f$ of $f_1$ and $f_2$  is a
  CMC surface. Moreover, the common
  $\mu$--Darboux transform is given by $\hat f = f_1 + \alpha\beta\invers$
  where   \begin{equation}
    \label{eq:commonDTparallel}
   \alpha=\alpha_2-\alpha_1\beta_1\invers \beta_2, \qquad
   \beta=\beta_2-\beta_1\alpha_1\invers
   \alpha_2\frac{a_2-1}{a_1-1}.
 \end{equation}  
\end{theorem}
Note that one could also allow $\mu_1=\mu_2$ and consider two
independent $d_{\mu_1}$--parallel sections,  but in this case the
resulting common CMC Darboux transform is $\hat f=  f$.

To investigate closed surfaces we recall:

\begin{definition}
  Let $f: M \to \R^3$ be a CMC surface and $d_\lambda$ its associated
  family of flat connections. For $\mu\in\C_*$ a $d_\mu$--parallel
  section $\alpha$ is called a \emph{section with multiplier} if  $\gamma^*\alpha=\alpha h_\gamma$ for all $\gamma\in\pi_1(M)$. Here
$h_\gamma: \pi_1(M) \to \C_*$ is a group homomorphism, and
$\gamma^*\alpha=\alpha\circ \gamma_\sharp$ where $\gamma_\sharp$ is
the associated deck transformation of $\gamma\in\pi_1(M)$.

A parameter $\mu\in\C_*$ is called a \emph{resonance point} if all
$d_\mu$--parallel sections have a multiplier.
\end{definition}

With this definition we have in the case of a CMC surface (for the
general case of a conformal immersion, see \cite{conformal_tori}):

\begin{prop}
 Let $f: M \to \R^3$ be a CMC surface and $\hat f = f+ T$ a
 $\mu$--Darboux transform given by a $d_\mu$--parallel section
 $\alpha$, $\mu\in\R\setminus\{0,1\}$. Then $\hat f$ is defined on $M$, rather than the
 universal cover $\tilde M$, if and only if $\alpha$ is a section
   with multiplier. 
\end{prop}

In particular, we can apply this to Bianchi permutability:

\begin{cor}
  Let $f: M \to\R^3$ be a CMC surface and $f_l: M \to\R^3$ be the CMC Darboux transforms
  given by $d_{\mu_l}$--parallel sections $\alpha_l$ with multipliers
  where $\mu_l\in\R\setminus\{0,1\}$, $l=1,2$, $\mu_1\not=\mu_2$. Then
  the common CMC
  Darboux transform $\hat f: M\to\R^3$ of $f_1$ and $f_2$ is defined on $M$ as well.
\end{cor}
\begin{proof}
  Let $h_l$ denote the multipliers of the  $d_{\mu_l}$--parallel
  sections $\alpha_l$.  
  By (\ref{eq:beta}) we see that $\gamma^*\beta_l = \beta_l
  (h_l)_{\gamma}$. Then  (\ref{eq:commonDTparallel}) shows that 
\begin{align*}
\gamma^*\alpha & = \gamma^*\alpha_2 - (\gamma^*\alpha_1)
(\gamma^*\beta_1)\invers (\gamma^*\beta_2) = \alpha_2 (h_2)_\gamma -
\alpha_1 (h_1)_\gamma(h_1)_\gamma\invers \beta_1\invers \beta_2
(h_2)_\gamma \\
&= \alpha (h_2)_\gamma\,.
\end{align*}
Therefore, $\alpha$ is a section with multiplier, and $\hat f =
f_1+\alpha\beta\invers$ is defined on $M$.
\end{proof}

\section{Delaunay bubbletons}\label{sect:dbubble}

We will now recall the construction  of (multi-) bubbletons  which are
CMC Darboux transforms at a resonance point of Delaunay surfaces, that is, of
CMC surfaces of revolution.

More generally, we will give all closed CMC Darboux transforms of
Delaunay surfaces in terms of Jacobi elliptic functions and
elliptic integrals.  To this end, we need to find all parallel sections of
the associated family of flat connections explicitly. Recall that surfaces
of revolution
can be conformally parametrised by
\[
  f(x,y) =ip(x) + j q(x)e^{-iy}\,, \qquad \text{ with }
  \quad (p')^2+(q')^2= q^2.
\] 
In the cases of Delaunay surfaces, $p, q$ are given by the Jacobi elliptic functions with parameter
$M=1-(1-\frac 1{1-r})^2$ where $r\in(-\infty, \frac 12)$ is the necksize and 
\[
q(x) = (1-r) \dn((1-r)x, M), \qquad p'(x) = q(x)^2 +r(1-r). 
\]
Note that  $r=\frac 12$ gives a cylinder, $r\in(0, \frac 12)$ unduloids
and $r\in(-\infty, 0)$ nodoids.

To find parallel sections $\alpha$ of a flat connection $d_\mu$ with $
\mu\in\R\setminus\{0, 1\}$, in the  associated family $d_\lambda$ of the
harmonic Gauss map $N$ of $f$, we need to solve
\[
  d\alpha = -\frac 12 df(N\alpha(a-1) + \alpha b),
\]
where $a=\frac{\mu+\mu\invers}2, b=i\frac{\mu\invers-\mu}2$. Here we
used (\ref{eq:H}) to express (\ref{eq:associated family}) in terms of
the CMC surface $f$. 
Since $f$ is conformal we have  $f_xN = -f_y$  and $f_yN = f_x$ so that
we need to solve the PDEs
\begin{align}
  \alpha_x &=\frac 12(f_y\alpha(a-1)-f_x\alpha b) \label{eq:alphax}, \\
    \alpha_y &= -\frac 12(f_x\alpha(a-1)+f_y\alpha b) \label{eq:alphay}.
\end{align}
We consider the second equation (\ref{eq:alphay}) first. We decompose $\alpha =\alpha_0 +
j\alpha_1$ and set $\tilde \alpha_1 = e^{iy} \alpha_1$ to obtain the linear system
\[
  \begin{pmatrix} \alpha_0 \\ \tilde\alpha_1 
  \end{pmatrix}_y = \frac 12\begin{pmatrix} -ip'(a-1) & iqb+ q'(a-1) 
    \\  iqb - q'(a-1) & i(2+p'(a-1))
  \end{pmatrix}
  \begin{pmatrix} \alpha_0\\ \tilde \alpha_1
  \end{pmatrix} =: A \begin{pmatrix} \alpha_0\\ \tilde \alpha_1
  \end{pmatrix}.
\]
The matrix $A$ has eigenvalues
$
  \frac i2(1\pm t)
$
where
\[t=\sqrt{1+2 r(1-r)(a-1)}.
\]

If $t=0$ then $\alpha_+ =\alpha_-$ gives only one
solution. In the case when $r = \frac 12$, we obtain $\mu=-1$, and the
general solutions are given by $\alpha = e^{\frac{iy}2} c$,
$c\in\H$. By Remark \ref{rem:double} the only Darboux transform in
this case is the parallel CMC
surface. For $r<\frac 12$ and $t=0$, the spectral parameter are
$\mu_b = \frac{r}{r-1}$ and $\mu_b\invers = \frac{r-1}r$, and  the matrix $A$ is not
diagonalisable. Therefore, we obtain a second independent  parallel section $\hat \alpha$ by the generalised
eigenvector, but  $\hat\alpha$ is not a section
with multiplier.  This shows that $\mu_b, \mu_b\invers$ are not
resonance points. (Indeed they are
 the branch points of the multiplier spectral curve of the
 unduloids and nodoids with necksize $r\not=\frac 12$.)

 Therefore, we can from now on assume
 $t\not=0$.  We obtain all $d_\mu$--parallel sections, 
 since the matrix has constant coefficients
in $y$ and $p'-q^2=r(1-r)$,  by the following fundamental sections 
\[
\alpha_\pm = e^{\frac {iy}2}\left(q b -i q'(a-1)+ j\left(1 +p'(a-1)\pm
    t\right)\right) e^{ \pm \frac {ity}2}c_\pm,
\]
where  $c_\pm$ are complex valued functions in $x$.

To solve for $c_\pm$,
we compute $(\alpha_\pm)_x$ in two ways. Since $c_\pm$ only
depend on $x$, we  can evaluate at $y=0$ and obtain
	\begin{align*}
		(\alpha_\pm)_x
			&= \left(q' b -i q'' (a-1) + jp''(a-1)\right)
                   c_\pm  \\ & \quad + \left(q b -i q'(a-1)+ j\left(1 +p'(a-1)\pm t\right)\right) c_\pm'.
	\end{align*}
        By (\ref{eq:alphax}) we also have at $y=0$
   \begin{align*}
	(\alpha_\pm)_x &=\frac 12(f_y\alpha_\pm(a-1) - f_x\alpha_\pm
                         b) \\
     &=\frac 12\Big(- i q \left(1 -2p'(a-1)\pm t\right) (a-1) 
       + q' \left(1 \pm t\right)b   \\  &\hspace{4cm}
			 +  j\big(p''(a-1) + i p' (1\pm t)b\big)  \Big) c_\pm\,,
   \end{align*}
   where we used $a^2+b^2=1$, $(p')^2+(q')^2 = q^2$ and $p'' = 2qq'$ since
   $p'=q^2+r(1-r)$. 
   By comparing the $\C$ and $j \C$ parts, respectively, we obtain 
   \begin{subequations}
		\begin{align*}
		(b q -i (a-1) q') c_\pm' &= \frac12 ( - i (a-1)q ( 1
                                           \pm t - 2 p'(a-1))  - bq'(1 \mp
                                           t) + 2 i (a-1) q'' )
                                           c_\pm, %\label{eqn:cpmeq1}
                  \\
		(1\pm t +(a-1)p') c_\pm' &= \frac12 \left(-(a-1)p'' + i(1 \pm t)b p'\right) c_\pm\,.\label{eqn:cpmeq2}
		\end{align*}
              \end{subequations}
              Up to multiplication by the factor $
             - \frac{b q  +
                                 i (a-1) q'}{ 1 \mp t + (a-1)
                                 p'}
                               $ these two equations are the same.
To solve for $c_\pm$, we note that when $1\pm t +(a-1)p' \neq 0$,
the second equation gives
	\begin{align*}
		\frac{c_\pm'}{c_\pm} 
			&= -\frac{1}{2}(\log  |1 \pm t + (a-1)p'|)' + \frac{i b (1\pm t)}{2 (a-1)} - \frac{i b (1\pm t)^2}{2 (a-1)} \frac{1}{1\pm t +(a-1)p'}.
	\end{align*}
Integrating both sides, we obtain
	\[
		\log|c_\pm| = -\frac{1}{2}\log  |1 \pm t + (a-1) p'|+ \frac{i b (1\pm t)}{2 (a-1)}x - \frac{i b (1\pm t)^2}{2 (a-1)} \int \frac{1}{1\pm t +(a-1)p'}\, dx.
              \]
     Abbreviating the Jacobi elliptic functions so that, for example,
     $\dn((1-r)x, M) = \dn$, we have $p' = q^2 + r(1-r)$ for $q =
     (1-r) \dn$. Thus, using $\dn^2 = 1 - M \sn^2$, one calculates
	\begin{align*}
		1\pm t +(a-1)p'
			&= (1\pm t  + (a-1)(1-r)) (1 - N \sn^2)
	\end{align*}
where $N = \frac{(a-1) (1-r)^2 M}{1\pm t  + (a-1)(1-r)}$. Therefore,
it remains to compute
$\int \frac{1}{1 - N \sn^2} \, dx,$ which is well-known to be
 the incomplete elliptic integral of the
 third kind
\begin{align*}
		\int \frac{1}{1 - N \sn^2} \, dx &= \frac{1}{1-r} \Pi(N; \am((1-r)x, M) \,|\, M)\,,
	\end{align*}
where $\am((1-r)x, M)$ is the Jacobi amplitude.

Observing that $\alpha_\pm$ are parallel sections, our results extend smoothly into the zeros of $1\pm t +(a-1)p'$. 
We summarise: 
\begin{theorem}
  \label{thm:parallel explicit}
Let $f(x,y) = ip + j q e^{-iy}$ be a Delaunay surface with necksize
$r$. For $\mu\in\R\setminus\{0, \pm 1\}$, let $a= \frac{\mu
  +\mu\invers}2$ and
\[
 t=\sqrt{1+2 r(1-r)(a-1)}\,.
\]
Then
all $d_\mu$--parallel sections for $t\not=0$  are explicitly given by
\[
\alpha = \alpha_+ m_+ + \alpha_- m_-
\]
with constants $m_\pm\in\C$, where
\[
\alpha_\pm = e^{\frac {iy}2}\left(q b -i q'(a-1)+ j\left(1 +p'(a-1)\pm
    t\right)\right) e^{ \pm \frac {ity}2}c_\pm
\]
and 
\[
c_\pm = \frac{\exp\left(\frac{i b (1\pm t)}{2(a-1)}\left( x -
      \frac{1\pm t}{1\pm t  + (a-1)(1-r)} \frac{1}{1-r} \Pi(
      \am((1-r)x,  M), M)\right)\right)}{\sqrt{|1 \pm t + (a-1) p'|}}\,.
	\]
Here $\Pi$ is the incomplete elliptic integral of the third kind and
$\am$ the Jacobi amplitude.
\end{theorem}

Since by Theorem \ref{thm:mudtcmc} all CMC Darboux transforms of
Delaunay surfaces are given algebraically by $d_\mu$--parallel
sections, we obtain also:

\begin{theorem}
The CMC Darboux transforms of a Delaunay surfaces can be expressed
explicitly in terms of Jacobi elliptic functions and the incomplete
elliptic integral of the third kind.
\end{theorem}

\begin{rem}
  The
parallel sections obtained here are special cases of the parallel
sections of a surface of revolution in \cite{isothermic_paper, sym-darboux} (after
adjusting the spectral parameter to account for the choice of dual
surface): we obtain exactly those classical Darboux transforms which
are CMC.
\end{rem}

Now, observe that
$\alpha_\pm$ are sections with multipliers
\[
  h_\pm =- e^{\pm \pi i\sqrt{1+2r(1-r)(a-1)}},
\]
and therefore, since $c_\pm$ are independent of $y$,  resonance points
have to satisfy 
\[
  1+2r(1-r)(a-1) = n^2\,, \quad n\in\N\,.
\]
% As discussed we can exclude the case $t=0$ which only gives rise to
% the parallel CMC surface in the case of $r=\frac 12$ or non-resonance
% points  in the case $r<\frac 12$.
Since $\mu\not=1$ and $t\not=0$ by assumption, the cases $n=0$ and
$n=1$  cannot occur.   Therefore, the (relevant) resonance points are given by
\[
  \mu_n^\pm =\frac{1}{2r(1-r)}\left(n^2-1 +2r(1-r) \pm
    \sqrt{(n^2-1)^2+ 4r(1-r)(n^2-1)}\right),
\]
for $n \in \N, n>1$.

We also recall that $\mu_n^\pm$ give rise to the same Darboux
transform by Remark \ref{rem:double}, allowing us to
consider only one of the solutions $\mu_n=\mu_n^+$.

The (non--trivial) CMC Darboux transforms of a Delaunay surface at resonance
points $\mu_n$  are called \emph{bubbletons}. Geometrically, the integer $n\in \N$ gives the number of lobes of the bubble on the
bubbleton (see Figures~\ref{fig:cylbubbleton} and \ref{fig:undbubbleton}).

\begin{figure}[H]
	\begin{minipage}{0.45\textwidth}
		\centering
		\includegraphics[width=\linewidth]{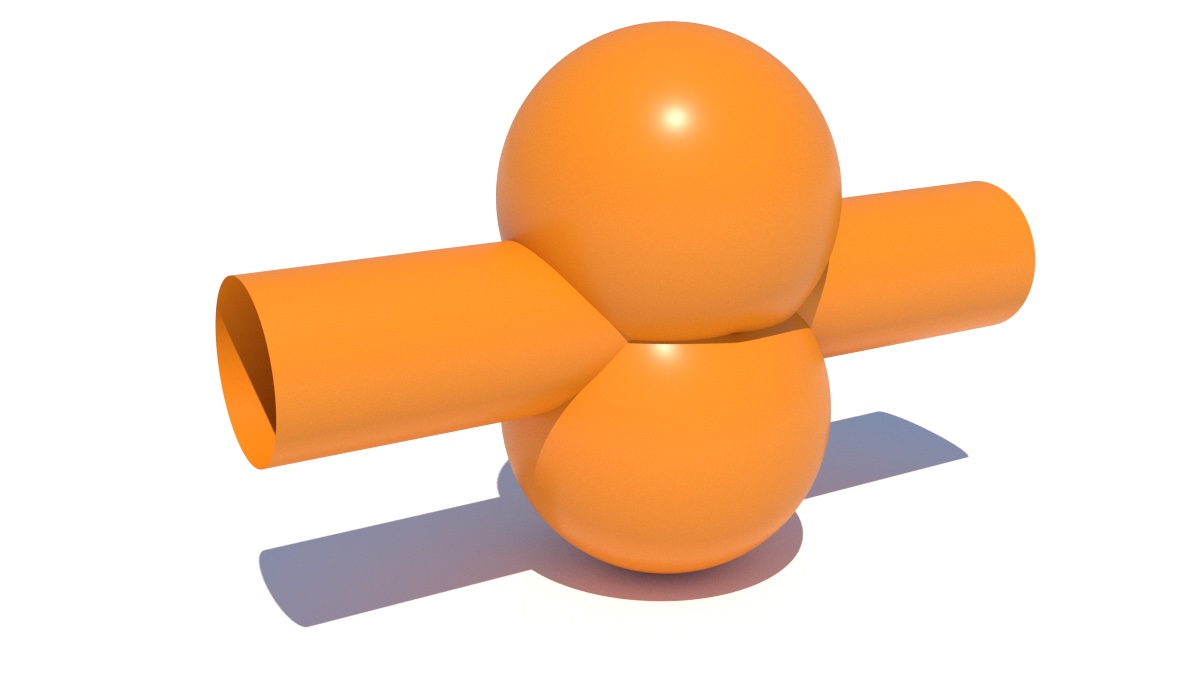}
	\end{minipage}\begin{minipage}{0.45\textwidth}
		\centering
		\includegraphics[width=\linewidth]{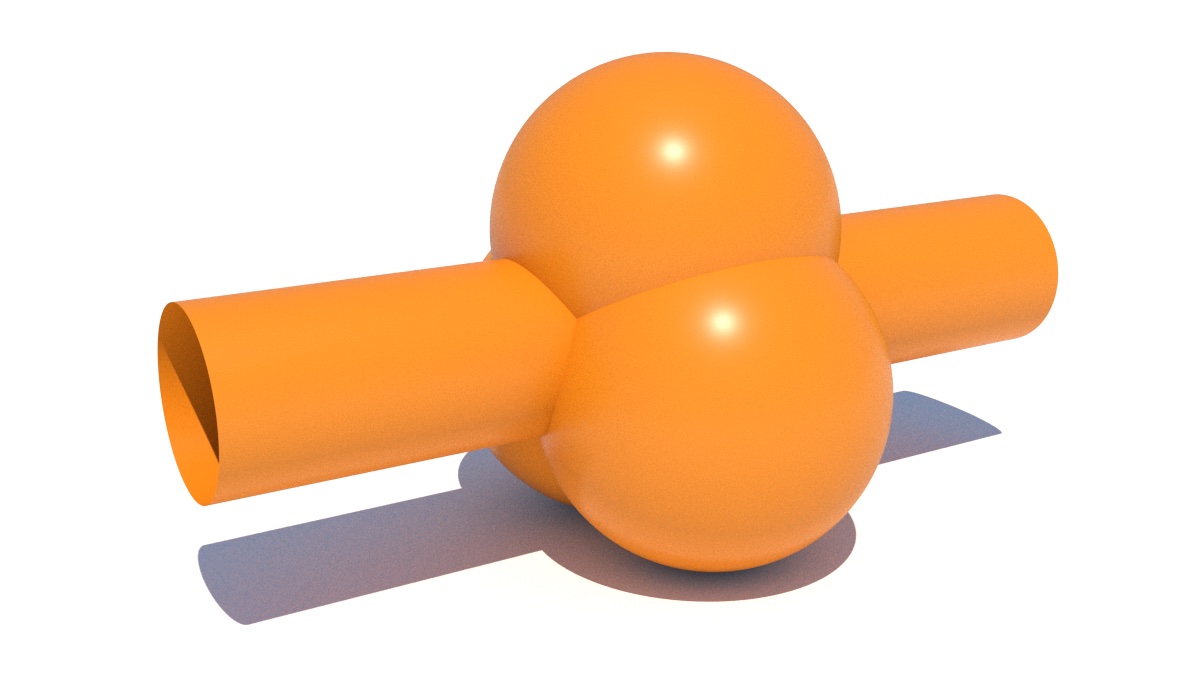}
	\end{minipage}
     \caption{Bubbletons of the cylinder, that is, $r=\frac 12$,  at resonance points $\mu_2=7+4\sqrt 3, \mu_3=
       17 +12\sqrt 2$.}
       \label{fig:cylbubbleton}
   \end{figure}

\begin{figure}[H]
	\begin{minipage}{0.5\textwidth}
		\centering
		\includegraphics[width=0.8\linewidth]{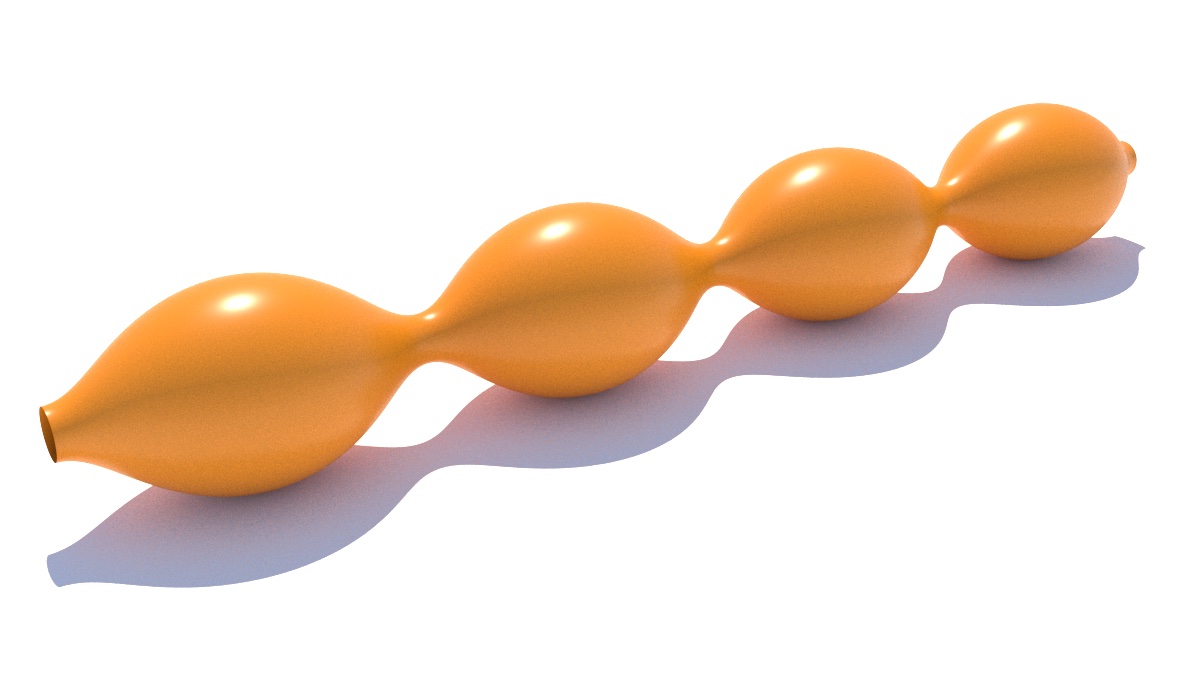}
	\end{minipage}
	\begin{minipage}{0.45\textwidth}
		\centering
		\includegraphics[width=\linewidth]{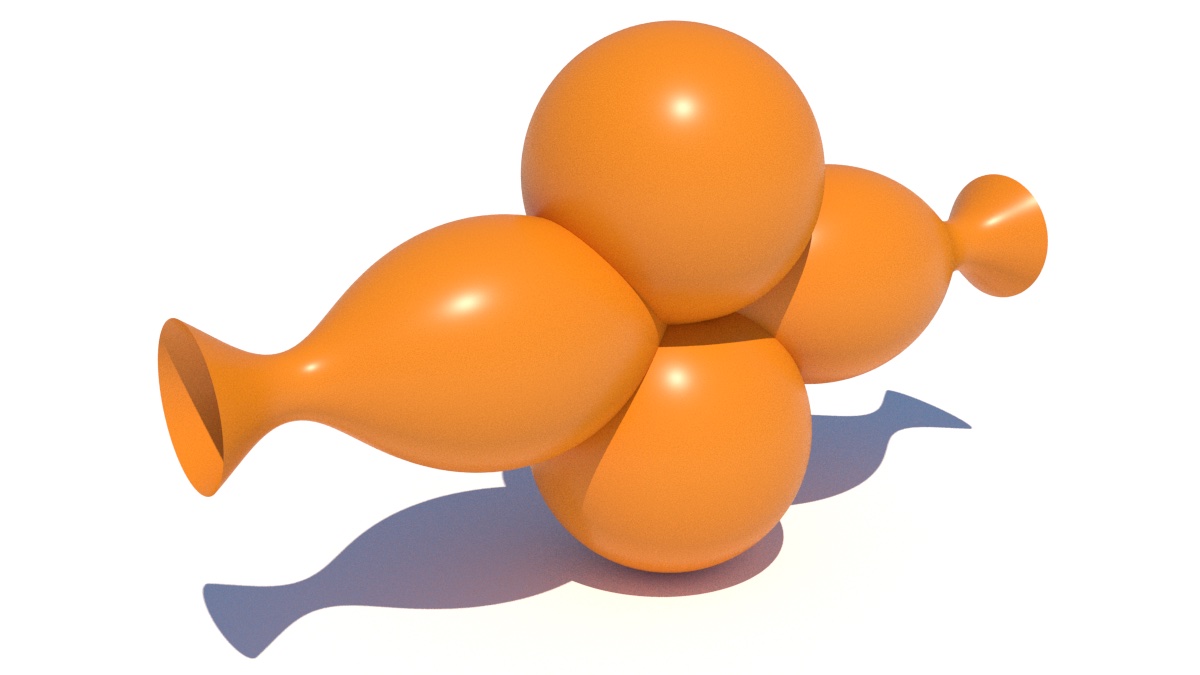}
	\end{minipage}\begin{minipage}{0.45\textwidth}
		\centering
		\includegraphics[width=\linewidth]{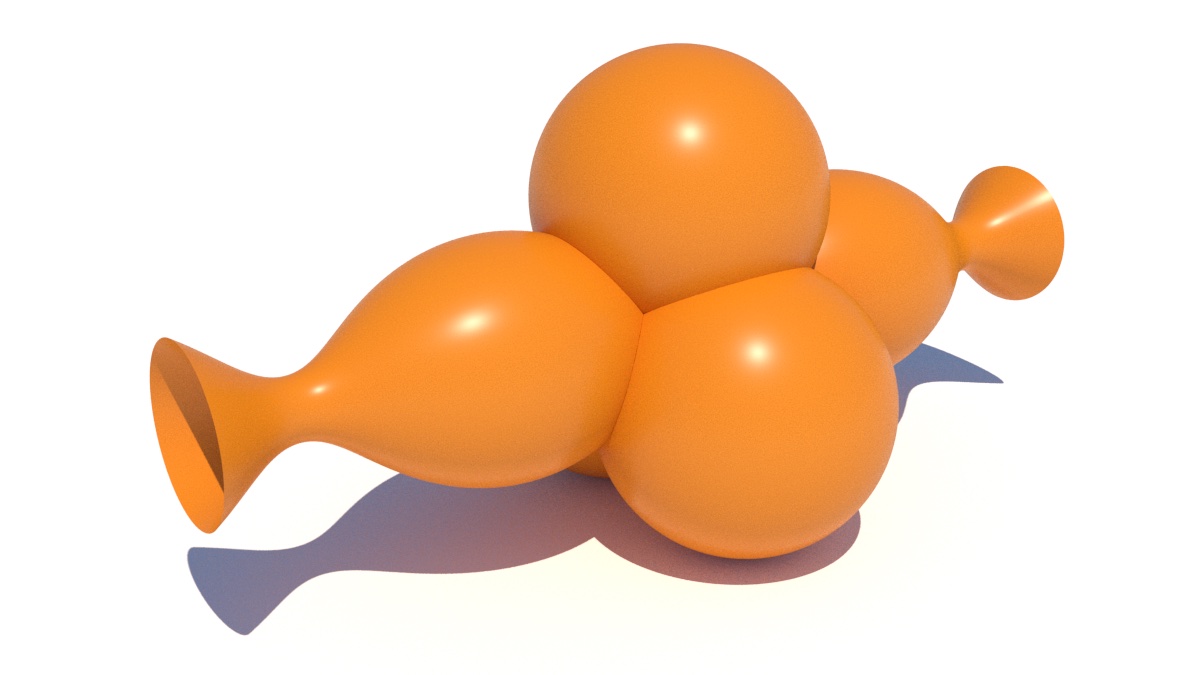}
	\end{minipage}
  \caption{An unduloid with $r=\frac 15$ and  its bubbletons with $n=2$ and $n=3$.}
  \label{fig:undbubbleton}
\end{figure}

Note that we require $\mu_n^\pm\in\R_*$ so that there is an additional condition on
$n$ in the case of a nodoid, see also \cite{KobayashiBubbletons}:
\[
\frac {1-n}2 \le  r< 0.
\]
The case when $n=1-2r$ gives $\mu=-1$ and the
corresponding CMC Darboux transform is the parallel CMC surface, that is,
we only obtain bubbletons for $\frac{1-n}2<r$ (see Figure~\ref{fig:nodbubbleton}). (In the case of the cylinder and unduloids, all
$\mu_n^\pm\in\R\setminus\{0, \pm 1\}$
for $n\in\N$, $n>1$.)

\begin{figure}[H]
	\begin{minipage}{0.45\textwidth}
		\centering
		\includegraphics[width=\linewidth]{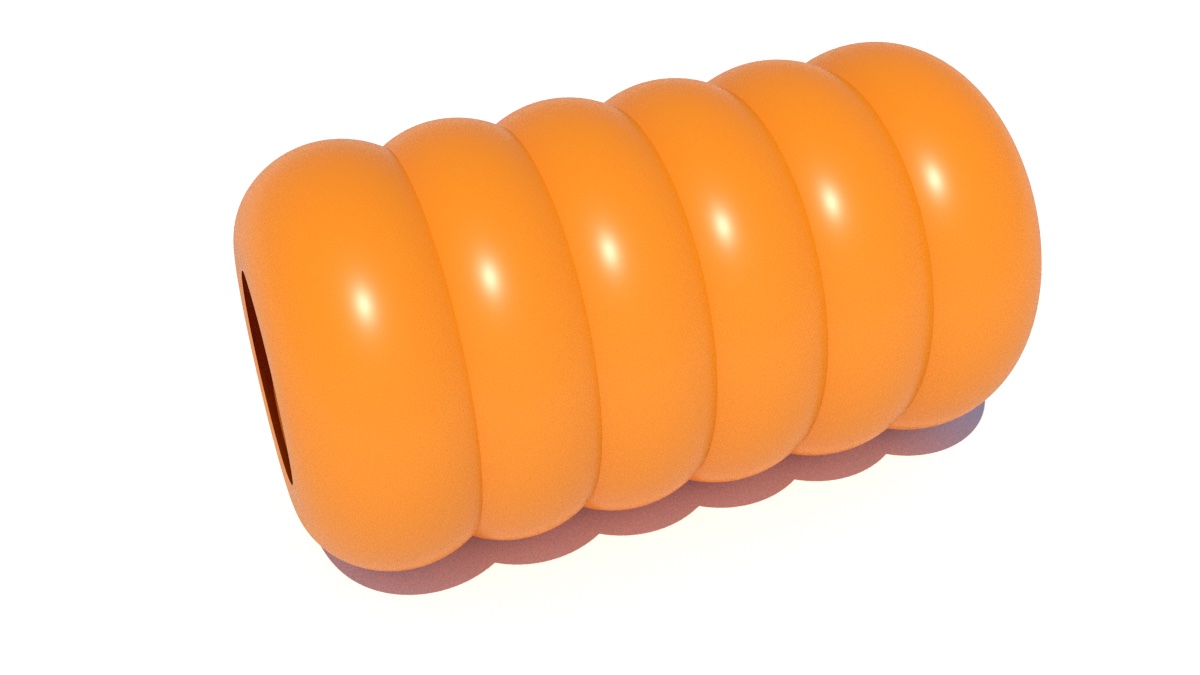}
	\end{minipage}
	\begin{minipage}{0.45\textwidth}
		\centering
		\includegraphics[width=\linewidth]{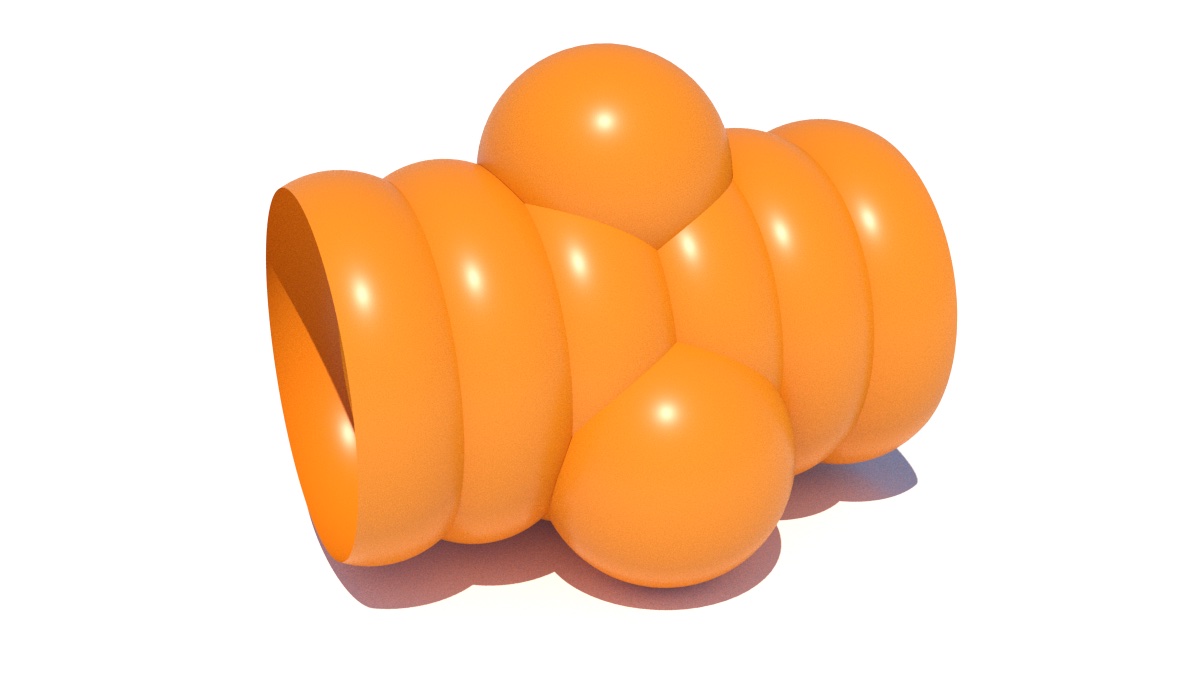}
	\end{minipage}
 
  \caption{A nodoid with necksize $r=-\frac 12$ and  a bubbleton with $n=3$: for $n=2$
    the parameter $\mu_2=-1$ and there is  no classical
    Darboux transform with 2 lobes for necksize $r=-\frac 12$.}
    \label{fig:nodbubbleton}
\end{figure}

% In the case of a cylinder, that is $r=\frac 12$ and
% \[
%   f(x,y) =\frac
%   12(ix + je^{-iy})
% \]
% we can solve (\ref{eq:alphax}) 
% \[
%   \alpha_x =\frac 12(f_y\alpha(a-1)-f_x\alpha b)
% \]
% explicitly and obtain
% \begin{align*}
%    \alpha_\pm&=e^{\frac {iy}2}\left(\pm\sqrt{1-a}+j (\sqrt 2\pm
%          \sqrt{1+a})\right)e^{\pm\frac i{2\sqrt 2}(x\sqrt{1-a} +y \sqrt{1+a}) } 
%          )  
% \end{align*} 
% which coincides with
% the parallel sections given in \cite{cmc}.  Moreover, we see that the
% parallel section obtained here are special cases of the parallel
% sections of a surface of revolution in \cite{sym-darboux} (after
% adjusting the spectral parameter to account for the choice of dual
% surface): we obtain exactly those classical Darboux transforms which
% are CMC.

   %  In particular, the resonance points of the round cylinder are given by
   %  \[
   %    \mu_n = 2n^2-1\pm 2n\sqrt{n^2-1}\,, \qquad n\in\Z\,,
   %  \]
   %  and the corresponding parallel sections are given by
   %  \[
   %   \alpha_\pm = \frac 1{\sqrt{2}} e^{\frac{iy}2}(\pm \sqrt{1-n^2} +
   %   j \sqrt 2(1 \pm n))e^{\pm\frac i{2\sqrt 2}(\sqrt{1-n^2}x+\sqrt 2 n y)}\,.
   % \]

 Finally, we can use Bianchi permutability to obtain
 multibubbletons. Given two parallel sections at distinct resonance
 points $\mu_n\not=\mu_l$, that is, $n\not=l$, we obtain two
 bubbletons with $n$ and $l$ lobes respectively. The common Darboux
 transform $\hat f$, a \emph{doublebubbleton},  of $f_n, f_l$ is a surface with two bubbles on it,
 where one bubble  has $n$ lobes and the other one has $l$ lobes (see Figure~\ref{fig:und23bubbleton}).

\begin{figure}[H]
	\includegraphics[width=0.5\linewidth]{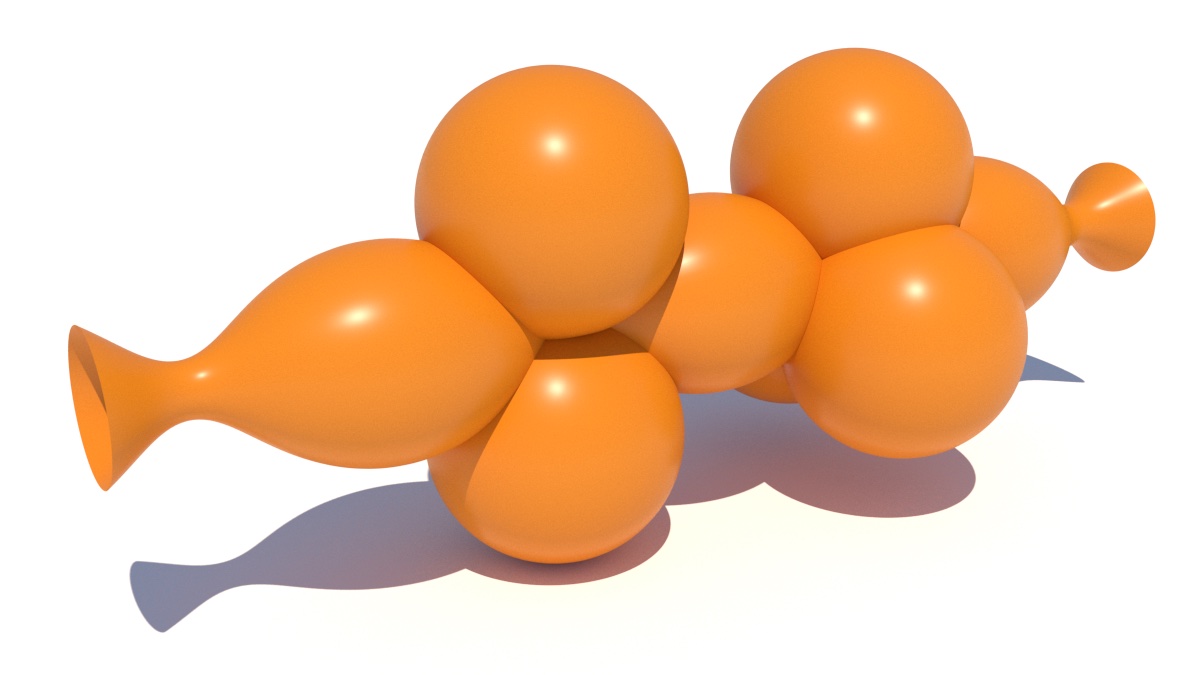}
	\caption{Delaunay doublebubbleton as a common Darboux transform of
		bubbletons with $n=2$ and $n=3$ lobes where $r=\frac 15$.}
	\label{fig:und23bubbleton}
\end{figure}

Applying Bianchi permutability one can now add an arbitrary number of
bubbles to the Delaunay surface. 
Since  Bianchi permutability needs distinct spectral parameter, we
 obtain  CMC multibubbletons with any
number of bubbles but every pair of bubbles has different numbers of
lobes.

On the other hand, in a recent paper \cite{sym-darboux} we showed that the
double Darboux transformation of a Delaunay surface which uses the
same spectral parameter twice (using a Sym--type argument) does not
give new closed CMC surfaces: 
In particular, we cannot obtain same-lobed CMC multibubbletons this
way.

\section{New CMC cylinders via the Darboux transformation}
\label{sect:newcmc}

In this section we will discuss CMC Darboux transforms of Delaunay
surfaces which close on a multiple cover. This is primarily motivated by our aim
to find same--lobed CMC multibubbletons but also in itself gives rise
to new interesting CMC surfaces.

In the case of a cylinder, Sterling and Wente \cite{wente_sterling}
discussed bubbletons on an $m$--fold cover via the Bianchi--B\"acklund
transform. In this case, the number of lobes $n$ must be bigger than
$m$. This naturally extends  the result for the single cover cylinder
where $n>1$.  We will recover this result in this section by investigating the
multipliers of 
parallel sections of a Delaunay surface on a multiple
cover. Surprisingly, in the case of unduloids and nodoids, additional
cases appear: in addition to the natural generalisations of the single
cover conditions $n>m$, we also obtain solutions for $n<m$. % In this case, the geometry of the resulting
% Darboux transforms look differently: there is no localised bubble as before
% but the number $n$ rather gives the dihedral symmetry of the surface. 

To this  end, we will first investigate the
resonance points of a multiple cover of a Delaunay surface, that
is, we will consider the surfaces
\[
  f(x,y) = ip(x) + jq(x) e^{-iy}
\]
on a $m$--fold cover, that is,  $x\in\R, y\in[0, 2m\pi], m\in \N$.   Then the parallel
sections for $\mu\in\R\setminus\{0,1\}$  are still given by Theorem \ref{thm:parallel explicit}
\[
\alpha_\pm^m = e^{\frac {iy}2}\left(q b -i q'(a-1)+ j\left(1 +p'(a-1)\pm
    t\right)\right) e^{ \pm \frac {ity}2}c_\pm
\]
for $t=\sqrt{1+2 r(1-r)(a-1)}\not=0$.

However, the resonance points and
multipliers change since we are now investigating the parallel
sections on an $m$--fold cover.  As before excluding 
$t=0$, the
multipliers of the parallel sections $\alpha_\pm^m$ are
\[
   h_\pm^m = (-1)^me^{\pm m\pi it} 
 \]
  which shows that the resonance points  $\mu_n^m$ have to satisfy
 \[
  1+2r(1-r)(a-1) =\frac{ n^2}{m^2}\,, \quad n\in\N_*,
\]
that is, the resonance points must be of the form
\[
  \mu_{n}^m=\frac{1}{2r(1-r)}\left(\frac{n^2-m^2}{m^2} +2r(1-r) \pm
  \sqrt{\frac{n^2-m^2}{m^2}\left(\frac{n^2-m^2}{m^2}+ 4r(1-r)\right)}\right).
\]
Note that $\alpha_\pm^m$ are only sections with multipliers if we
consider them as sections on a $m$--fold cover. That is, we cannot
construct Darboux transforms with period $2\pi \tau$ with $\tau>0$
this way if $\tau\not\in \Z$.  Moreover, we can assume without loss of
generality that $(m,n)$
are co--prime since $\mu_{n}^m =\mu_{n'}^{m'}$ for $\frac{n}{m} = \frac{n'}{m'}$.

Additionally, $(m,n)$ must  satisfy the condition
\[
  \frac{n^2-m^2}{m^2}\left(\frac{n^2-m^2}{m^2}+ 4r(1-r)\right) >0,
\]
since we consider only real spectral parameter
$\mu\in\R\setminus\{0,\pm 1\}$.
% (We will return to the case of complex spectral parameter in a future
% paper). \todo{can we at least look at them!?!}

Note that by Theorem \ref{thm:mudtcmc} the resulting  $\mu$--Darboux
transforms for spectral parameter $\mu_n^m\in\R\setminus\{0, \pm 1\}$
are CMC surfaces in 3--space. Thus we have extended the result for
round cylinders in \cite{wente_sterling} (obtained via Bianchi's
B\"acklund transformation) to general Delaunay surfaces:

\begin{theorem}\label{thm:fullBif}
  Let $f(x,y) = ip + jqe^{-iy}$ be a Delaunay surface, $r\in(-\infty,
  \frac 12)$, $r\not=0$,
  $M= 1-(1-\frac 1{1-r})^2$, $q(x) = (1-r)\dn((1-r)x,
  M), p'(x) = q^2+r(1-r)$. Let $(m, n)$ with co--prime $m,n \in\N_*$, be an \emph{admissible
  pair}, that is,
   
\[ 
\begin{array}{@{}r@{\quad\quad}l@{\qquad}l@{}}
\text{for } r< 0: &	m < n, \,\frac{m-n}{2m} < r &\text{or}\qquad m
                                                      > n\\[.2cm]
\text{for } 0<r<\frac 12:& 		m < n & \text{or}\qquad m > n, \, r < \frac{m-n}{2m} \\[.2cm]
\text{for } r= \frac 12: &		m < n \,.&  
    \end{array}
\]

\begin{comment}

\[ 
\left\lbrace\,
\begin{array}{@{}l@{\quad}l@{\qquad \quad}l@{}}
	m < n, \,\frac{m-n}{2m} < r &\text{or}\quad m > n,  & \text{for } r \in (-\infty, 0), \\
		m < n & \text{or}\quad m > n, \, r < \frac{m-n}{2m}, & \text{for } r \in (0, \frac{1}{2}), \\
		m < n, &  &\text{for } r = \frac{1}{2}\,. \\
    \end{array}
\right.\]

\[\begin{cases}
		m > n \quad\text{or}\quad m < n, \,\frac{m-n}{2m} < r, & \text{for } r \in (-\infty, 0), \\
		m < n \quad\text{or}\quad m > n, \, r < \frac{m-n}{2m}, & \text{for } r \in (0, \frac{1}{2}), \\
		m < n, & \text{for } r = \frac{1}{2}\,. \\
              \end{cases}\]
          \end{comment}
          
Then every Darboux transform with parameter
  $\mu_n^m$ 
  is a closed CMC surface on the
$m$--fold cover of the Delaunay surface. Put differently,   for each admissible
pair $(m,n)$ the spectral parameter
  \[
    \mu_n^m= \frac{1}{2r(1-r)}\left(\frac{n^2-m^2}{m^2} +2r(1-r) \pm
  \sqrt{\frac{n^2-m^2}{m^2}\left(\frac{n^2-m^2}{m^2}+
      4r(1-r)\right)}\right)
  \]
  is a resonance point and $\mu_n^m \in\R\setminus\{0,\pm 1\}$.
\end{theorem}

Note that \cite{wente_sterling} only considered the case of cylinders
via the Bianchi--B\"acklund transformation, and thus, only considered
the case when $n>m$. However, in the case of unduloids and nodoids,
interesting new CMC cylinders appear for $n<m$. We discuss the three
cases separately, assuming that $(m,n)$ is an admissible pair:

{\bf Cylinder.} In the case of a cylinder, $r=\frac 12$, we have
  $\mu_n^m\in\R\setminus\{0,1\}$
  if and only if $n>m$. The resulting Darboux transforms
  have a localised bubble with $n$ lobes and close
on the $m$--fold cover of the Delaunay surface. These are the bubbletons
  given in \cite{wente_sterling} (see Figure~\ref{fig:cylmulcover}).
  
\begin{figure}[H]
	\begin{minipage}{0.45\textwidth}
		\centering
		\includegraphics[width=\linewidth]{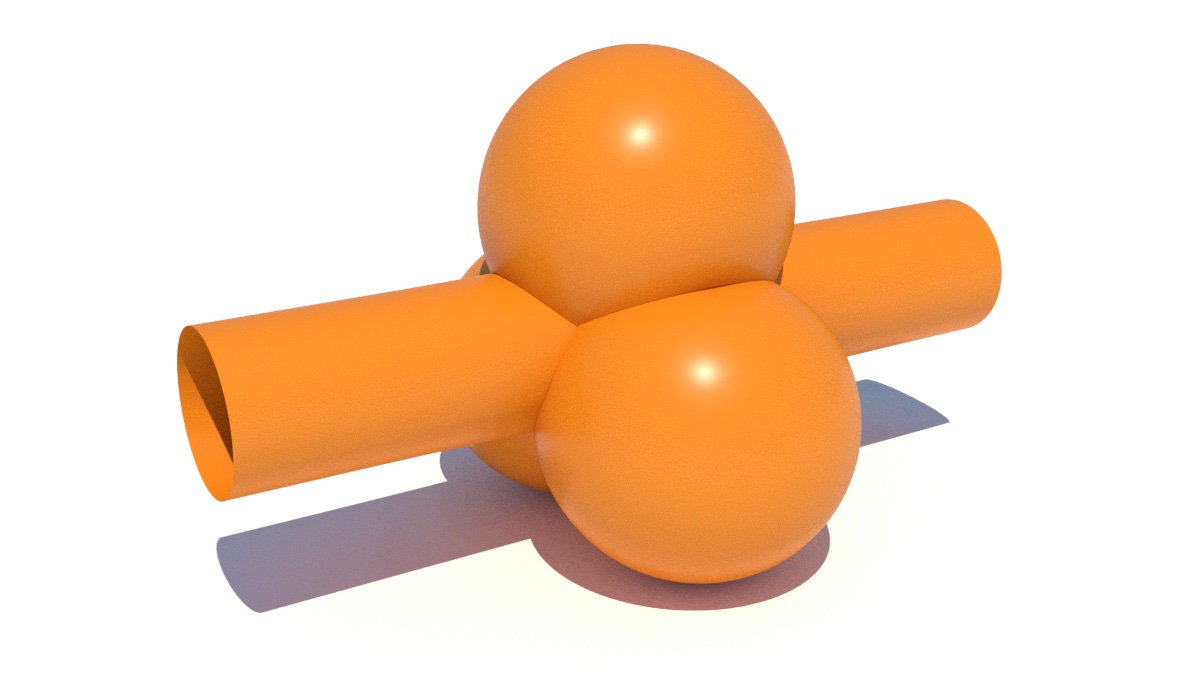}
	\end{minipage}
	\begin{minipage}{0.45\textwidth}
		\centering
		\includegraphics[width=\linewidth]{cylinder3.jpg}
	\end{minipage}
	\caption{Darboux transform of the cylinder  with $n=3$ on the 2--fold cover, next to a
		bubbleton with $n=3$ on the single cover.}
	\label{fig:cylmulcover}
\end{figure}

{\bf Unduloid.} In the case of an unduloid we have $0<r <\frac 12$ and
  $4r(1-r) \in  (0,1)$.  Thus, for $n>m$ we still have
  $\mu_n^m\in\R_*$. In this case, we obtain Darboux transforms
  with a localised bubble with $n$ lobes as in the case of cylinders (see Figure~\ref{fig:undmulcover}).

  \begin{figure}[H]
	\begin{minipage}{0.45\textwidth}
		\centering
		\includegraphics[width=\linewidth]{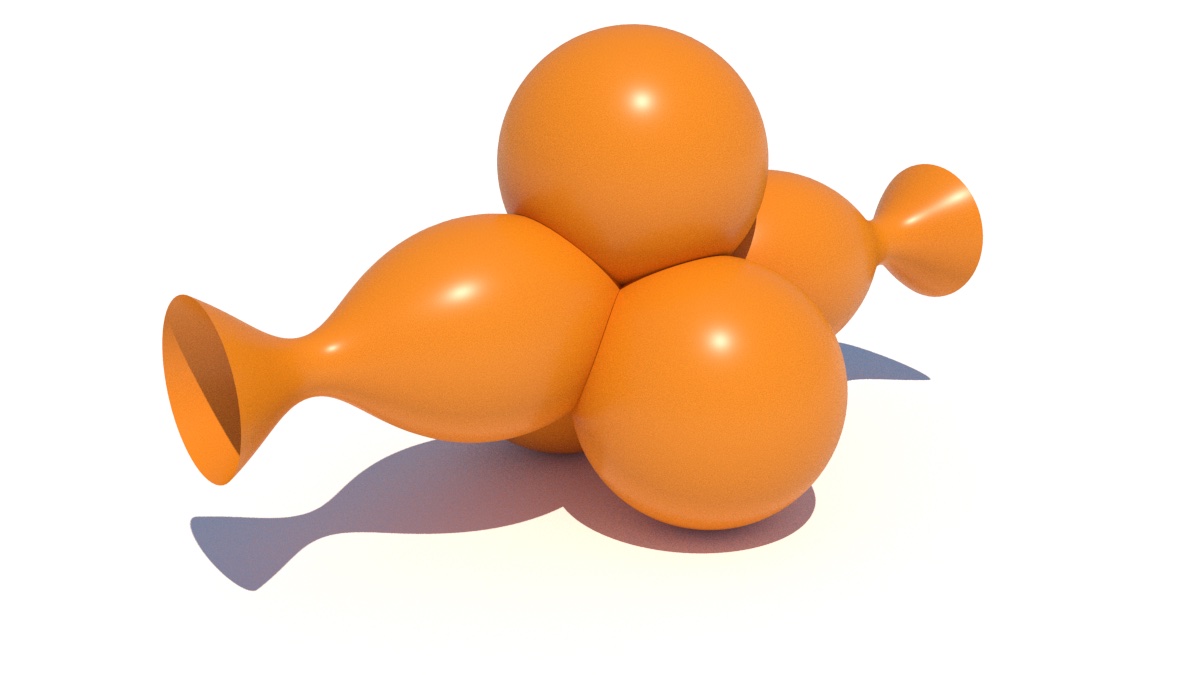}
	\end{minipage}
	\begin{minipage}{0.45\textwidth}
		\centering
		\includegraphics[width=\linewidth]{unduloid3.jpg}
	\end{minipage}
	\caption{Darboux transform of an unduloid  with necksize $r=\frac 15$,
	$n=3$ on the 2--fold cover, next to a
	bubbleton with $n=3$ on the single cover.}
	\label{fig:undmulcover}
  \end{figure}

  However, for unduloids, we obtain additional solutions which could not be
  observed in \cite{wente_sterling}:   if $n<m$ and $r <
  \frac{m-n}{2m}$ we also obtain closed CMC surfaces (see Figure~\ref{fig:undmulcover2}). Note that in
  this case, the ``bubble'' of the surface is less localised, and the
  number $n$ gives the dihedral symmetry of the surface rather than
  the number of lobes.

\begin{figure}[H]
	\includegraphics[width=0.8\linewidth]{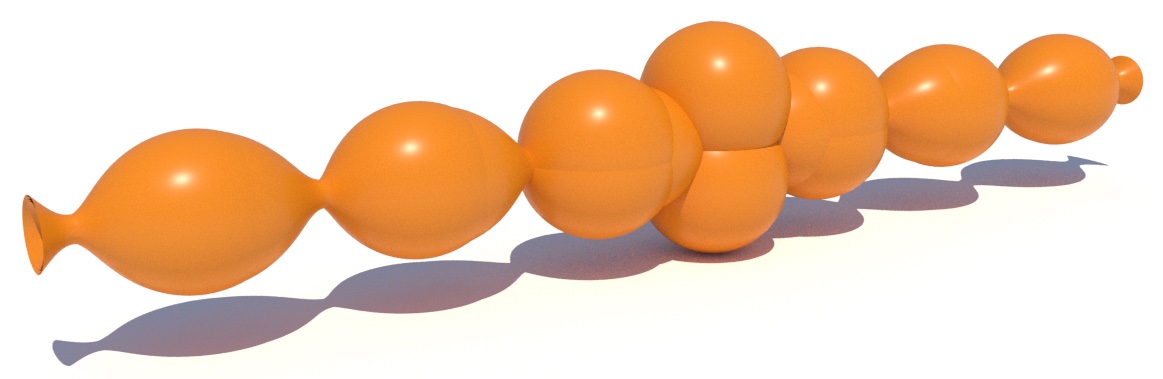}
	\caption{Darboux transform with $n=2$, $m=7$, of an
		unduloid with necksize $r=\frac 15$.}
	\label{fig:undmulcover2}
\end{figure}

{\bf Nodoid.} In the case of a nodoid, we have $r<0$, and as in the case
  of the single cover, we obtain an obstruction for the number $n$ of
  lobes on the bubbleton: in this case we have a bubbleton with a
  localised bubble with $n>m$ lobes on the $m$--fold cover if
  $\frac{m-n}{2m} >r$ (see Figure~\ref{fig:nodmulcover}).

\begin{figure}[H]
	\begin{minipage}{0.45\textwidth}
		\centering
		\includegraphics[width=\linewidth]{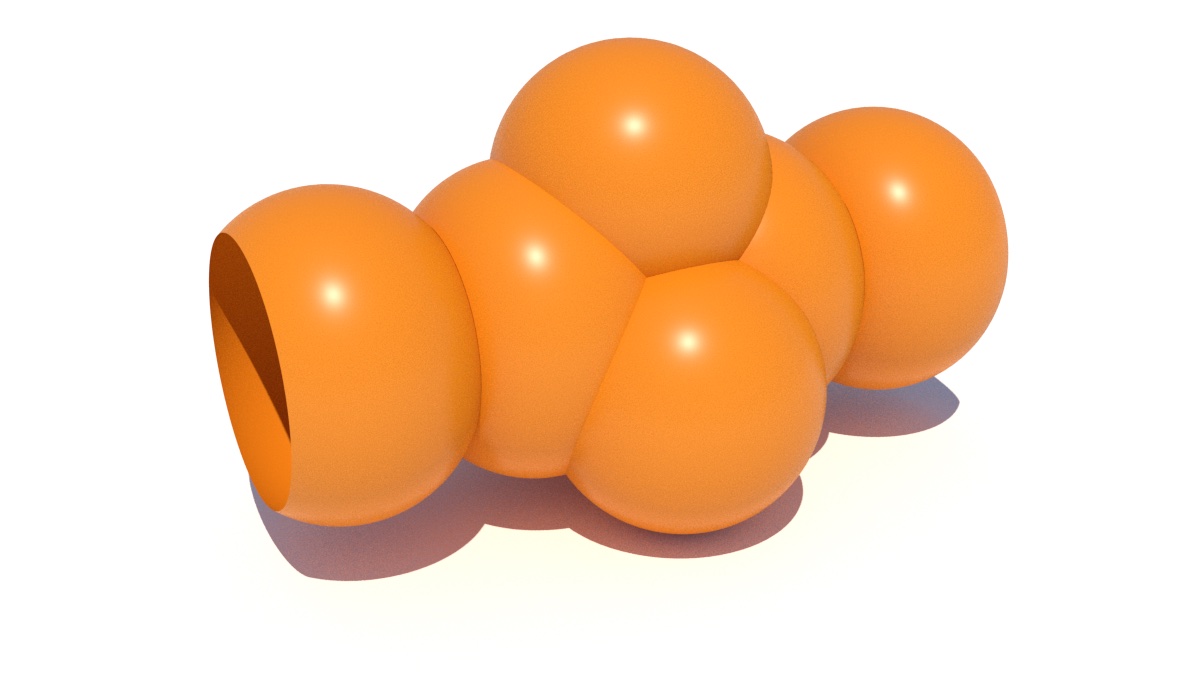}
	\end{minipage}
	\begin{minipage}{0.45\textwidth}
		\centering
		\includegraphics[width=\linewidth]{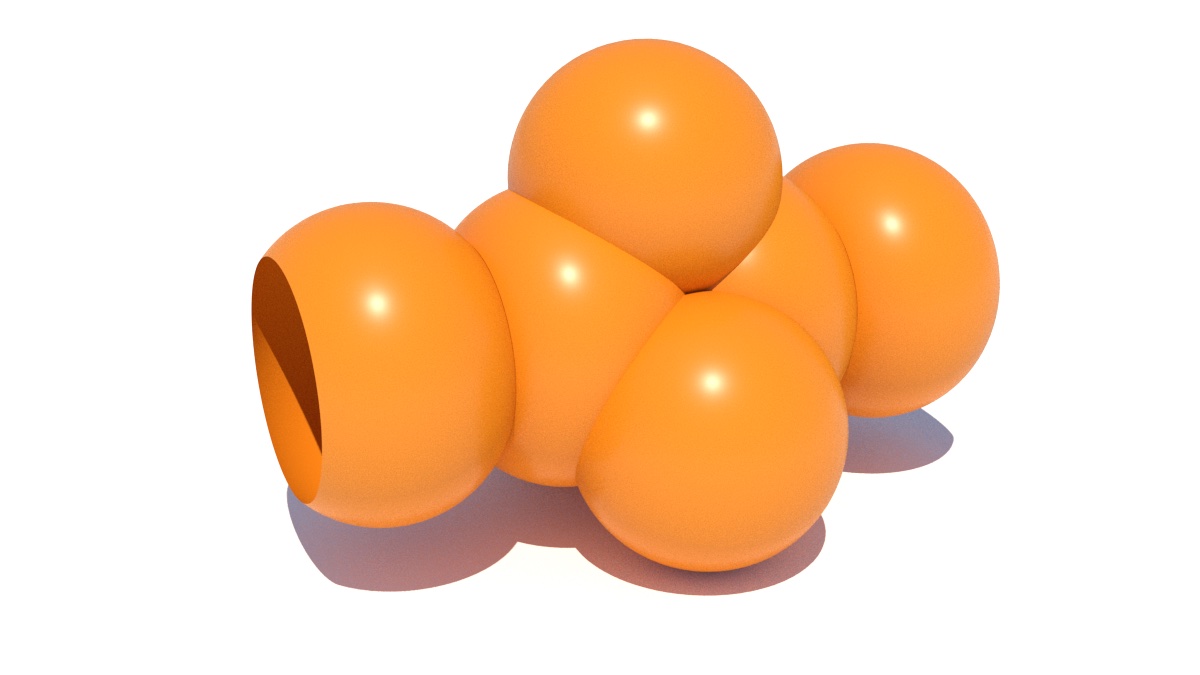}
	\end{minipage}
	\caption{Darboux transform of a nodoid with necksize $r=-\frac 18$,  $n=3$ and $m=2$, next to
		a bubbleton with $n=3$ on the single cover.}
	\label{fig:nodmulcover}
\end{figure}

In the case of a nodoid, every choice of $n<m$ gives a bubbleton on
the $m$--fold cover. In this case, the number $n$ again gives the
dihedral symmetry of the surface.

\begin{figure}[H]
	\begin{minipage}{0.45\textwidth}
		\centering
		\includegraphics[width=\linewidth]{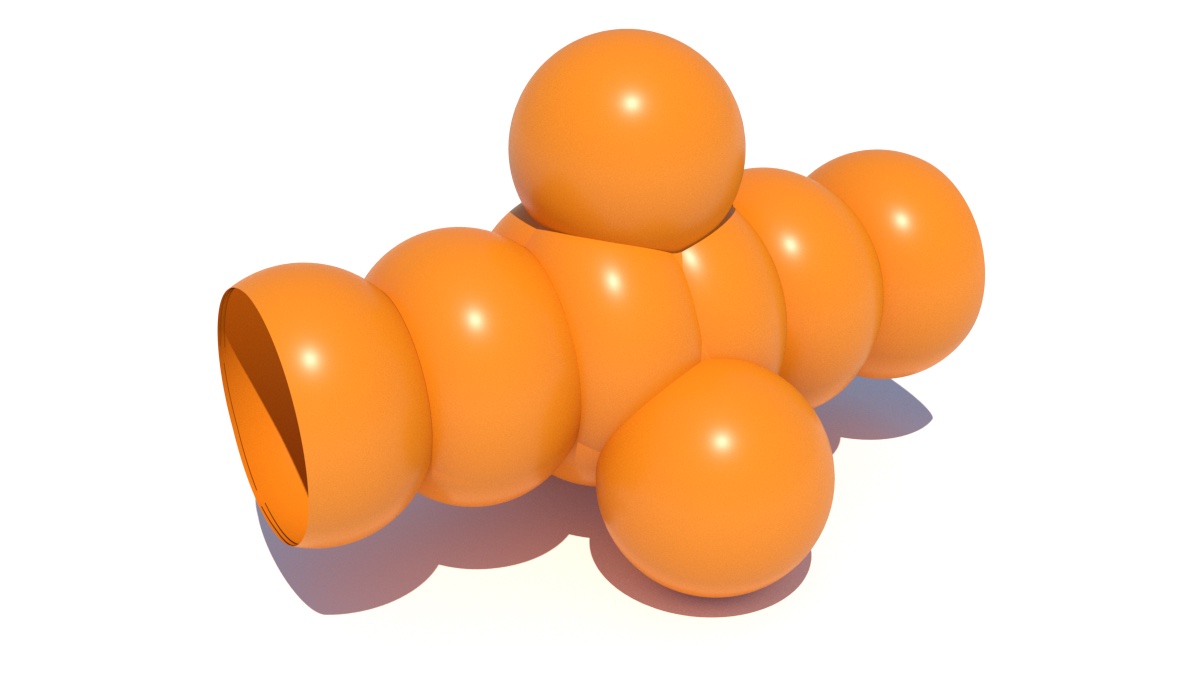}
	\end{minipage}
	\begin{minipage}{0.45\textwidth}
		\centering
		\includegraphics[width=\linewidth]{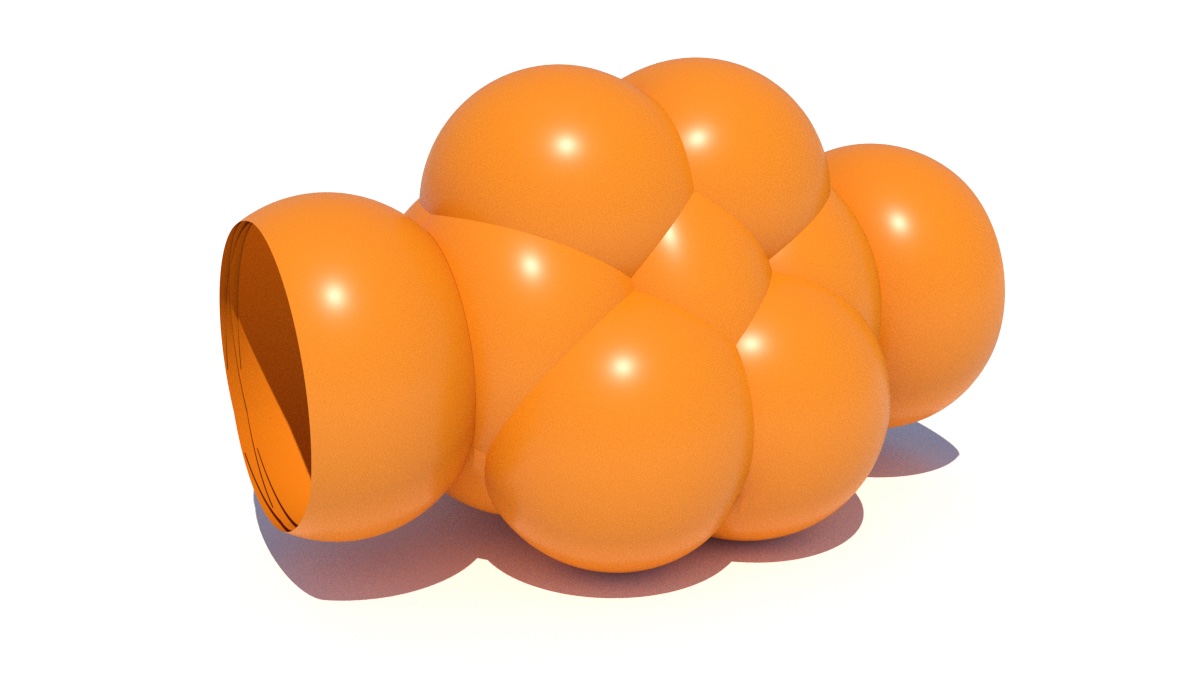}
	\end{minipage}
\caption{Darboux transform of a nodoid with necksize $r=-\frac 18$,
  $n=3$ and $m=4$ with two different
  initial conditions.}
  \label{fig:nodmulcover2}
\end{figure}

In particular, we can also use $n=1$ (see Figure~\ref{fig:nod1}):
\begin{figure}[H]
	\includegraphics[width=0.5\textwidth]{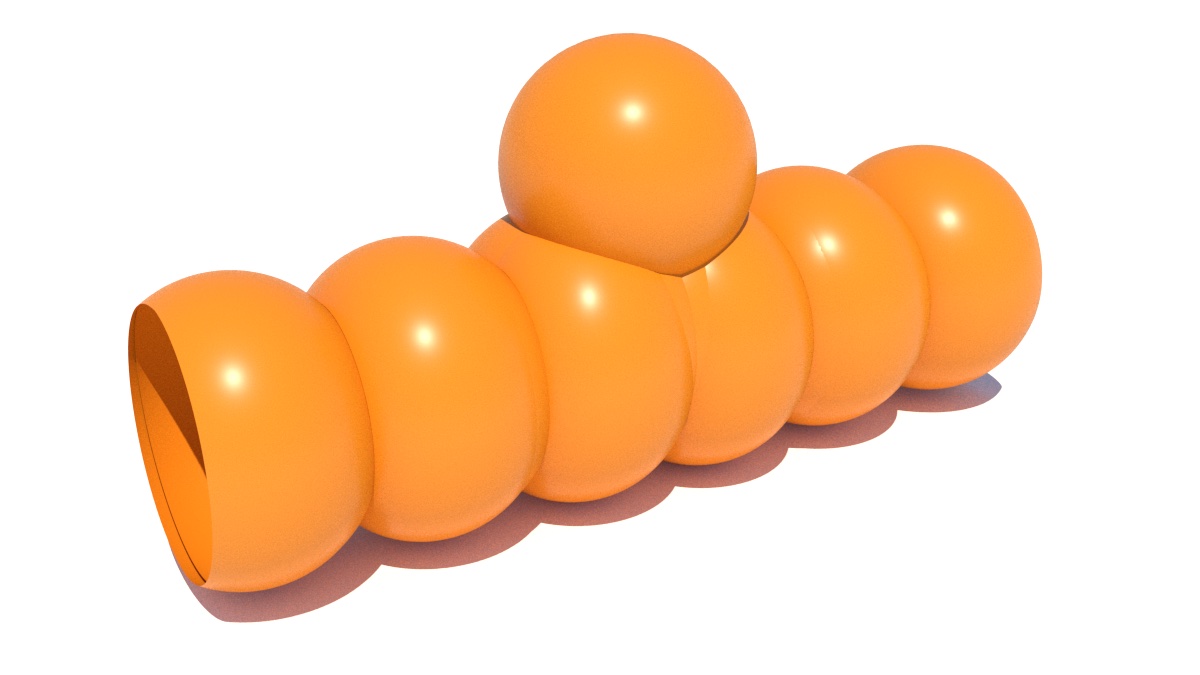}
	\caption{Darboux transform of a nodoid with necksize $r=-\frac 18$,  $n=1$ and $m=2$.}
	\label{fig:nod1}
\end{figure}
We will return to a more detailed analysis of these new examples of
closed CMC Darboux transforms of nodoid and unduloids in a
future paper.

\section{Same--lobed CMC multibubbletons}
\label{sect:same}

Changing our point of view, we now fix the number of lobes $n$ and
consider all possible $m$ such that $\mu_n^m$ is a resonance point.
To obtain same--lobed CMC multibubbletons (with localised bubbles) we
restrict here to the case $n>m$: for such an admissible pair $(m,n)$ we
then obtain a bubbleton with $n$ lobes on an $m$--fold cover, and each
$\mu_n^m\in\R_*$ is a resonance point.

Recall that Bianchi permutability requires  different spectral parameters
$\mu_1\not= \mu_2$ to obtain new surfaces as common Darboux transform
of the two surfaces $f_1$ and $f_2$. 
Using now resonance points $\mu_n^{m_1}\not= \mu_n^{m_2}$ given by two admissible pairs $(m_1, n)$ and $(m_2, n)$ with
$m_1\not=m_2$,  we obtain by Bianchi permutability a CMC doublebubbleton $\hat f$  which has two bubbles with
$n$ lobes each. Note that $\hat f$ closes on the $L$--fold cover of
the Delaunay surface where $L$ is the least common multiple
of $m_1$ and $m_2$.   Repeating the procedure we obtain further
multibubbletons as long as there are distinct admissible pairs
$(m,n)$, $m<n$. Denoting by
\[
  \varphi(n) =  \#\{ m\in \N \mid m,n \, \text{co--prime}, \, 1\le m<n, \, \frac{m-n}{2m}<r\} ,
\]
the number of admissible pairs with $m<n$, (for cylinders and
unduloids,  this is just Euler's totient
function $\varphi$ of $n$),
we obtain our final result (see also Figure~\ref{fig:5555}).

\begin{theorem}\label{thm:samecmc}
Let $f$ be a Delaunay surface with necksize $r$ and 
  denote by $\varphi(n)$ the number of admissible pairs $(m,n)$ with $m<n$. 
 Then for $1\le
l\le \varphi(n)$  there exists a CMC
multibubbleton with $l$ bubbles, each with $n$ lobes, which closes  on a 
multiple cover of the Delaunay surface. 
\end{theorem}

\begin{figure}[H]
	\begin{minipage}{0.49\textwidth}
		\centering
		\includegraphics[width=\linewidth]{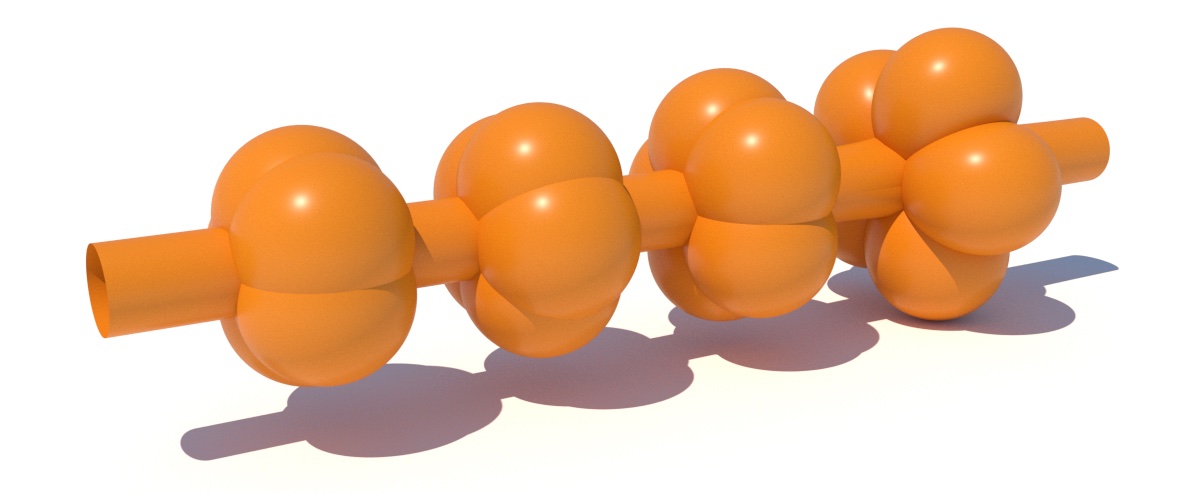}
	\end{minipage}
	\begin{minipage}{0.49\textwidth}
		\centering
		\includegraphics[width=\linewidth]{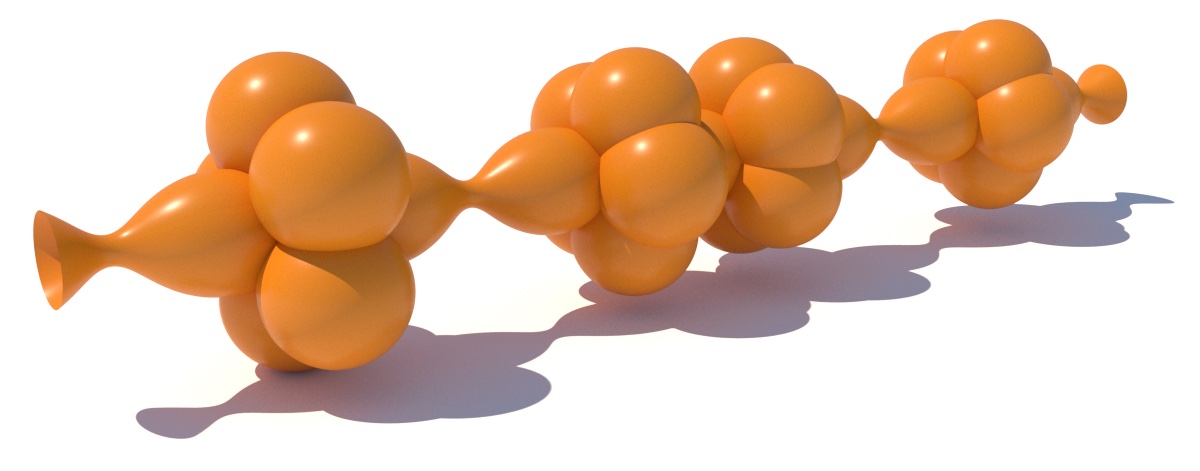}
	\end{minipage}
\caption{CMC same--lobed multibubbletons of a cylinder and unduloid  given via Bianchi permutability:
  here $n=5$, $\varphi(5)=4$ and all  available resonance points $\mu_5^1, \mu_5^2,
  \mu_5^3$ and $\mu_5^4$ are used. The multibubbletons close on the 12--fold
  cover and have 4 bubbles with 5 lobes each. }
  \label{fig:5555}
\end{figure}


\end{document}